\documentclass[11pt,a4paper]{amsart}
\usepackage[english]{babel}
\usepackage{pst-grad} 
\usepackage{pst-plot} 
\usepackage{pstricks}
\usepackage{amsmath,amssymb,mathrsfs,amsthm}
\usepackage[dvips]{graphicx}
\usepackage{epsfig}

\newcommand{\RR}{\mathbb R}

\newcommand{\TT}{\mathbb T}
\newcommand{\pat}{\partial_t}
\newcommand{\pax}{\partial_x}

\newcommand{\jeps}{\mathcal{J}_\epsilon*}

\newcommand{\vertiii}[1]{{\left\vert\kern-0.25ex\left\vert\kern-0.25ex\left\vert #1 
    \right\vert\kern-0.25ex\right\vert\kern-0.25ex\right\vert}}

\usepackage[pdfauthor={Rafael Granero-Belinchon},pdftitle={...},bookmarks,colorlinks]{hyperref}

\newcounter{comentcount}
\newcounter{teocount}
\newtheorem{lem}{Lemma}
\newtheorem{prop}{Proposition}

\newtheorem{teo}[teocount]{Theorem}  
\newtheorem{defi}{Definition}

\newtheorem{remark}{Remark}

\title{On the generalized Buckley-Leverett equation}

\author[J. Burczak]{Jan Burczak}
\email{jb@impan.pl}
\address{Institute of Mathematics of the Polish Academy of Sciences, Warsaw, 21 00-956, Poland}

\author[R. Granero-Belinch\'{o}n]{Rafael Granero-Belinch\'{o}n}
\email{rgranero@math.ucdavis.edu}
\address{Department of Mathematics, University of California, Davis, CA 95616, USA}

\author[G. K. Luli]{Garving K. Luli}
\email{kluli@math.ucdavis.edu}
\address{Department of Mathematics, University of California, Davis, CA 95616, USA}

\begin{document}

\begin{abstract}In this paper we study the generalized Buckley-Leverett equation with nonlocal regularizing terms. One of these regularizing terms is diffusive, while the other one is conservative.  We prove that if the regularizing terms have order higher than one (combined), there exists a global strong solution for arbitrarily large initial data. In the case where the regularizing terms have combined order one, we prove the global existence of solution under some size restriction for the initial data. Moreover, in the case where the conservative regularizing term vanishes, regardless of the order of the diffusion and under certain hypothesis on the initial data, we also prove the global existence of strong solution and we obtain some new entropy balances. Finally, we provide numerics suggesting that, if the order of the diffusion is $0< \alpha<1$, a finite time blow up of the solution is possible.
\end{abstract}

\maketitle 

\tableofcontents

\section{Introduction}
In this paper we study the case of the Buckley-Leverett equation with generalized regularizing terms provided by fractional powers of the laplacian $(-\Delta)^{\alpha/2}=\Lambda^\alpha$
\begin{equation}\label{BL}
\pat u+\pax \left[\frac{u^2}{u^2+M(1-u)^2}\right]=-\nu\Lambda^\alpha u-\mu\Lambda^\beta \pat u,\, x\in\Omega,\,t>0, \end{equation}
with initial data
$$
0\leq u(x,0)=u_0(x)\leq 1,
$$
and where $M>0$ is a fixed constant. Here $\Omega$ is either $\Omega=\RR$ or $\Omega=\TT$.

Let us immediately emphasize that $u_0(x)\leq 1$ is not a smallness condition, since, in applications, $u$ denotes a certain proportion (compare the following literature outline). Equation \eqref{BL} is a nonlocal regularization of the classical Buckley-Leverett equation
\begin{equation}\label{BLstandard}
\pat u+\pax \left[\frac{u^2}{u^2+M(1-u)^2}\right]=0,\, x\in\Omega,\,t>0, 
\end{equation}
The nonlinearity in equation \eqref{BL} is regularized in two different ways: firstly, due to the diffusive term 
$$
-\nu\Lambda^\alpha u,
$$
and secondly due to the conservative term
$$
-\mu\Lambda^\beta \pat u.
$$

Equation \eqref{BLstandard} was derived by Buckley \& Leverett in\cite{buckley1941mechanism} and it has been well studied since then (see LeVeque \cite{leveque1992numerical} and  Mikeli\'c \& Paoli \cite{mikelic1997derivation}). This equation is used to describe a two-phase flow in a porous medium. For example, oil and water flow in soil or rock. In this situation $u$ represents the saturation of water and $M>0$ is the water-over-oil viscosity ratio. Equation \eqref{BLstandard} is a prototype of conservation laws with convex-concave flux functions (see for instance Lax \cite{lax1957hyperbolic}, Glimm \cite{glimm1965solutions}, Hong \cite{hong2006extension}, Hong \& Temple \cite{temple2004bound}, Bayada, Martin \& Vazquez \cite{bayada2006generalized}). Under the effect of dynamic capillarity, \eqref{BLstandard} needs to be modified with two regularizing terms (see Hassanizadeh \& Gray \cite{hassanizadeh1990mechanics, hassanizadeh1993thermodynamic}):
\begin{equation}\label{BLstandard2}
\pat u+\pax \left[\frac{u^2}{u^2+M(1-u)^2}\right]=\nu \pax^2u+\nu^2 \tau \pax^2\pat u,\, x\in\Omega,\,t>0. 
\end{equation}
Equation \eqref{BLstandard2} has been studied by many authors. For instance, Van Duijn, Peletier \& Pop \cite{van2007new} derived existence conditions for solution of travelling wave form. Moreover, this leads to admissible shocks for \eqref{BLstandard}, which violate the Oleinik entropy condition. In \cite{HongWuYuan}, Hong, Wu \& Yuan proved the global existence of classical solution to \eqref{BLstandard2} with $u_0\in H^1(\RR)$. Furthermore, they proved that the solution becomes $C^\infty((0,\infty)\times \RR)$. For equation \eqref{BLstandard}, Hong, Wu \& Yuan proved that if the total variation $T.V.(u_0)$ is sufficiently small, then there exists a solution to \eqref{BLstandard}. Wang \& Kao \cite{wang2011bounded} studied \eqref{BLstandard2} on a finite interval $\Omega=(0,L)$ and showed that the solution, $u_L,$ converges to the solution $u_\infty$ on the half line as $L\rightarrow\infty.$
\subsection{Aim and outline}
The purpose of this paper is to study \eqref{BL}. We are mainly interested in global existence of solution together with their qualitative behaviour as well as in the finite time singularities.

 We provide details of our results in the following subsection. Subsection \ref{ssec:not} contains notation, including the definition of a weak solution and certain preliminaries. Section \ref{sec2} provides new entropy inequalities for the fractional laplacian that are interesting by themselves, therefore these inequalities are stated for an arbitrary dimension $d$. Sections  \ref{sec2}-\ref{sec:8} contain proofs of our results. Finally, in Section  \ref{sec9}, we provide some numerical results suggesting the existence of finite time singularities for the cases $0<\alpha<1$ and $\mu=0$. These numerics suggest also  that in the critical case $\alpha=1$ the solution exists globally. This is in agreement with the results for the Burgers equation with fractional dissipation by Kiselev, Nazarov \& Shterenberg, \cite{kiselevburgers} and Dong, Du \& Li \cite{dong2009finite}. Let us remark that, when the term $\mu\Lambda^\beta \pat u$ is added to the equation, even for $\alpha=\beta=0.25$ there is no evidence of blow-up. Consequently, our numerics appear to discard a finite time blow-up scenario when $\mu>0$.

 To the best of our knowledge, all our results are new.

\subsection{Results}
First, let us provide a result concerning the global existence of weak solutions for \eqref{BL}, corresponding to rough initial data, \emph{i.e.}, merely $0\leq u_0\leq 1$ a.e., as well as concerning new entropy balances (that are needed in the existence part of the result, but are interesting by themselves).

\begin{prop}\label{teo1}
Let $0\leq u_0\leq 1$, $u_0\in L^1(\Omega)\cap L^\infty(\Omega)$ be the initial data for \eqref{BL} with $\nu>0$, $0<\alpha<2$, $\mu=0$ and $M>0$. Then there exists a global weak solution such that
$$
u\in L^\infty(0,\infty;L^1(\Omega)\cap L^\infty(\Omega))\cap L^2(0,\infty;H^{\alpha/2}(\Omega)).
$$
Furthermore, if $u$ is an $L^2(0,T;H^1(\Omega))$ solution to \eqref{BL}, then the following entropy inequalities hold
\begin{equation}\label{eq:enA}
\int_{\Omega}u(t)\log(u(t))dx+\nu\int_0^t\int_{\Omega}\Lambda^\alpha u(s)\log(u(s))dxds  \le
\int_{\Omega}u_0\log(u_0)dx,
\end{equation}
and
\begin{align}\label{eq:enB}
\int_{\Omega}\bigg{[}u(t) & +\frac{M}{1+M}\bigg{]}\log\left(u(t)^2+M(1-u(t))^2\right)dx\\
&\quad-\frac{2M^{1.5}}{1+M}\int_\Omega\arctan\left(\sqrt{M}\left(\frac{1}{u(t)}-1\right)\right)dx\nonumber\\
&\quad+\nu\int_0^t\int_\Omega\Lambda^\alpha u(s)\bigg{[}\log\left(u(s)^2+M(1-u(s))^2\right)\nonumber\\
&\quad+\frac{2(M+1)u(s)^2}{u(s)^2+M(1-u(s))^2}\bigg{]}dxds\nonumber\\
&\quad  \le \int_{\Omega}\left[u_0+\frac{M}{1+M}\right]\log\left(u_0^2+M(1-u_0)^2\right)dx\nonumber\\
&\quad-2\frac{M^{1.5}}{1+M}\int_\Omega\arctan\left(\sqrt{M}\left(\frac{1}{u_0}-1\right)\right)dx.\nonumber
\end{align}
\end{prop}
Let us remark that the terms
$$
\int_0^t\int_{\Omega}\Lambda^\alpha u(s)\log(u(s))dxds,
$$
and
$$
\int_0^t\int_\Omega\Lambda^\alpha u(s)\bigg{[}\log\left(u(s)^2+M(1-u(s))^2\right)+\frac{2(M+1)u(s)^2}{u(s)^2+M(1-u(s))^2}\bigg{]}dxds
$$
will provide a $L^2_t$ bound on a fractional derivative of the solutions. The proof of Proposition \ref{teo1}  will be established in Section \ref{sec3}.

Our next results concern the qualitative behaviour of smooth solutions. 
In case $\Omega = \TT$, we denote the average of $u_0$ by
$$
\langle u_0\rangle=\frac{1}{2\pi}\int_\TT u_0(x)dx.
$$

\begin{prop}\label{teo1b}
Let $u$ be the classical solution to \eqref{BL} with initial data $0\leq u_0\leq1$, where $\nu>0$, $0<\alpha\leq 2$ and $M>0$. Then,
\begin{enumerate}
\item for $\Omega=\TT$:
$$
\|u(t)\|_{L^\infty(\TT)}\leq \langle u_0\rangle+(\|u_0\|_{L^\infty}-\langle u_0\rangle)e^{-\frac{2\Gamma(1+\alpha)\cos((1-\alpha)\pi/2)}{\pi^{1+\alpha}}t}.
$$
\item for $\Omega=\RR$:
$$
\|u(t)\|_{L^\infty}\leq \frac{\|u_0\|_{L^\infty}}{\left(1+\alpha \frac{\Gamma(1+\alpha)\cos((1-\alpha)\pi/2)}{2\pi}\|u_0\|_{L^\infty}^{\alpha}t \right)^{\frac{1}{\alpha}}}.
$$
\end{enumerate}
\end{prop}
Our main results  address the problem of global existence of smooth solutions. More precisely, we have results for three cases, depending on the values of the parameters $\alpha$ and $\beta$.
\begin{enumerate}
\item \emph{the subcritical case: the higher space derivative is in the dissipative term, i.e.} $1<\max\{\alpha,\beta\}\leq 2$. Here we show global existence of smooth solutions with no restrictions on the initial data. Compare Theorem \ref{teo2}.
\item  \emph{the critical case: the transport term exactly balances the regularizing terms, i.e.} $1=\max\{\alpha,\beta\}$. Here, for  $\mu> 0$ and $\beta=1$, we  prove global existence of smooth solutions \emph{without any size restriction on the initial data}. In the other cases we need certain smallness conditions. Namely, for $\mu=0$, $\nu>0$ and $\alpha=1$, we obtain global existence of smooth solutions for initial data satisfying a smallness restriction on the lower order norm $L^\infty$; this smallness restriction is explicit in terms of $\nu$ and $M$. Finally, in the case $\mu> 0$, $\alpha=1$ and $0<\beta<1$,  we obtain the global existence for initial data satisfying a smallness condition in $H^{\frac{1+\beta}{2}}$. The smallness restriction is here slightly less explicit, but easily computable. See Theorem \ref{teo3}.
\item \emph{the supercritical case: the higher space derivative is in the transport term, i.e.} $0\leq \alpha<1$ and $\mu=0$. Even here, for $\Omega = \TT$, we are able to  prove global existence of smooth solutions for smooth, periodic initial data satisfying an explicit smallness restriction on the Lipschitz norm $W^{1,\infty}$.
\end{enumerate}

The remaining open problems are in the critical and supercritical regime. In particular, our results do not apply to the case where
$$
\max\{\alpha,\beta\}<1,\;\mu>0,
$$
and there is no large data, global results for the critical case with $\mu=0, \nu>0, \alpha=1$. In the context of the latter, let us observe that on one hand, there are certain new methods available for nonlinear problems with nonlocal critical dissipation, like the method of moduli of continuity by Kiselev, Nazarov \& Shterenberg \cite{kiselevburgers}, the fine-tuned DeGiorgi method by Caffarelli \& Vasseur \cite{CafVas} or the method of the nonlinear maximum principles by Constantin \& Vicol \cite{cv} (see also Constantin, Tarfulea \& Vicol \cite{ctv}). But on the other hand, our nonlinearity is more complex than the typical ones.

Now, let us state the main theorems. First, we study the subcritical case $\max\{\alpha,\beta\}>1$.
\begin{teo}\label{teo2}
Let $0\leq u_0\leq 1$, $u_0\in H^s(\Omega)$ be the initial data for \eqref{BL} with $M>0$. Then \eqref{BL} has 
a global solution
$$
u\in C([0,T],H^s(\Omega))\cap L^2(0,T;H^{s+\alpha/2}(\Omega))\;\text{ for all } 0<T<\infty.
$$
Moreover, for $t \le T$, the solution satisfies
$$
\|u(t)\|^2_{L^2}+\mu\|u(t)\|^2_{\dot{H}^{\beta/2}}+2\nu\int_0^t\|u(s)\|^2_{\dot{H}^{\alpha/2}}ds=\|u_0\|^2_{L^2}+\mu\|u_0\|^2_{\dot{H}^{\beta/2}},
$$
provided
\begin{itemize}
\item \emph{either} $\nu>0, 1<\alpha\leq2, \mu=0$ and $s\geq1$ (purely dissipative regularization),
\item \emph{or} $\mu>0$, $\nu\geq0$, $1<\max\{\alpha,\beta\}\leq2$ and $s\geq\max\{1+\frac{\beta}{2}\}$ (dissipative-conservative regularization).
\end{itemize}
\end{teo}
For the critical case, let us define the following constants 
\begin{defi}
Let $\gamma^*$ be a constant such that 
$$
\frac{2\gamma^*(M+1)}{M}+\frac{2(\gamma^*)^2(\gamma^*+M)(M+1)^2}{M^2}=\nu,
$$
and let $\gamma$ be any fixed number such that $0<\gamma<\gamma^*$. 

Next, let $C_S$ be the Sobolev's constant corresponding to the embedding 
$$
H^{\frac{1+\beta}{2}}\hookrightarrow L^\infty.
$$
\end{defi}
We have
\begin{teo}\label{teo3}
Let $0\leq u_0\leq 1$, $u_0\in H^s(\Omega)$, $s\geq1$ be the initial data for \eqref{BL} with $M>0$. Then \eqref{BL} has a global solution
$$
u(t)\in C([0,T],H^s(\Omega))\cap L^2(0,T;H^{s+0.5}(\Omega))  \quad \forall \;{T < \infty}
$$
that satisfies the energy balance
$$
\|u(t)\|^2_{L^2}+\mu\|u(t)\|^2_{\dot{H}^{\beta/2}}+\nu\int_0^t\|u(s)\|^2_{\dot{H}^{\alpha/2}}ds=\|u_0\|^2_{L^2}+\mu\|u_0\|^2_{\dot{H}^{\beta/2}}.
$$
Under the following conditions:
\begin{itemize}
\item[(i)] \emph{Either:} $\nu\geq0$,  $0\leq\alpha\leq 2$ and $\mu>0$, $\beta=1$ (conservative regularization with no smallness conditions on the data).
\item[(ii)] \emph{Or:}  $\nu > 0$, $\alpha=1$, $\mu=0$ and the initial data is such that
\begin{equation}\label{eqsmall}
\|u_0\|_{L^\infty}\leq \gamma.
\end{equation}
In this case the solution satisfies the maximum principle
$$
\|u(t)\|_{H^{0.5}}\leq \|u_0\|_{H^{0.5}}.
$$
\item[(iii)] \emph{Or:} $\nu>0$,  $\alpha=1$, $\mu>0$,  $0<\beta<1$ and the initial data is such that
\begin{equation}\label{eqsmall2}
\|u_0\|_{L^2}^2+(1+\mu)\|u_0\|_{\dot{H}^{0.5}}^2+\mu\|u_0\|^2_{\dot{H}^{\frac{1+\beta}{2}}}\leq (1+\mu)\frac{\gamma^2}{C^2_S}.
\end{equation}
Then, the solution satisfies the maximum principle
$$
\|u(t)\|_{\dot{H}^{0.5}}^2+\mu\|u(t)\|^2_{\dot{H}^{\frac{1+\beta}{2}}}\leq \|u_0\|_{\dot{H}^{0.5}}^2+\mu\|u_0\|^2_{\dot{H}^{\frac{1+\beta}{2}}}.
$$
\end{itemize}

\end{teo}

\begin{remark}
The fact that
\begin{equation}\label{boundM}
\left\|\frac{2u}{u^2+M(1-u)^2}-\frac{u^2(2u-2M(1-u))}{\left(u^2+M(1-u)^2\right)^2}\right\|_{L^\infty}\leq C(M),
\end{equation}
independently of the value of $u$, allows for a global result relying on a condition related to $M$, $\nu$ and $\mu$. However, we are interested in results that deal with every possible value of the physical parameters present in the problem.
\end{remark}

In our opinion, there are two reasons, at least in the case $\mu=0$, why the smallness condition \eqref{eqsmall} may be seen as a rather mild restriction. The first one is that the size restriction affects a lower norm, merely $L^\infty$, keeping the higher seminorms as large as desired. The second one is that, given $M$ and $\nu$, the constant $\gamma^*$ can be easily computed. For instance, if we further assume $\gamma^*\leq1$, the expression for $\gamma^*$ is explicit:
$$
\gamma^*=\min\left\{1,\frac{-1+\sqrt{1+2(M+1)\nu}}{1+M}\right\}.
$$

In particular, $\gamma^*=O(\nu)$ for $\nu\ll1$.

The last case, namely  where $0<\alpha<1$, is harder because the leading term in the equation is the transport term. However, under certain conditions, we can prove the global existence of solutions. Before we can state the relevant result, we need some notation.  
\begin{defi}\label{def:Sig} Let $\gamma$ and $M$ be given, positive constants. Define $\Sigma(\gamma)$ as follows
\begin{align}
\Sigma(\gamma)& =\frac{2\gamma(M+1)}{M}+\frac{4\gamma(\gamma+M)(M+1)^2}{M^2}+\frac{4\gamma^2(\gamma+M)(M+1)^2}{M^2}\nonumber\\
&\quad+\frac{2\gamma^3(M+1)^3}{M^2}+\frac{8\gamma^3(\gamma+M)^2(M+1)^3}{M^3}\label{17}.
\end{align}
Next,  let $\gamma^*$ be a small enough constant such that 
\begin{equation}\label{12}
\Sigma(\gamma^*)\leq \nu\frac{\Gamma(1+\alpha)\cos((1-\alpha)\pi/2)}{\pi}.
\end{equation}
\end{defi}
Then we have the following
\begin{teo}\label{teo4}
Let $\Omega = \TT$ and $0\leq u_0\leq 1$, $u_0\in H^s$ with $s\geq2$ be the initial data for \eqref{BL}, where $M>0$ and $\nu>0$, $0<\alpha<1$,  $\mu=0$. Fix any $0<\gamma<\gamma^*$, with $\gamma^*$ in accordance with Definition \ref{def:Sig}. If the initial data is such that
$$
\|u_0\|_{W^{1,\infty}(\TT)}\leq \gamma,
$$
then there exists a global solution
$$
u\in C([0,T],H^s(\TT))\cap L^2(0,T;H^{s+\alpha/2}(\TT)) \quad \forall \;{T < \infty}.
$$
Furthermore, this solution satisfies the maximum principle
$$
\|u(t)\|_{W^{1,\infty}(\TT)}\leq \|u_0\|_{W^{1,\infty}(\TT)},
$$
and the energy balance
$$
\|u(t)\|^2_{L^2}+\nu\int_0^t\|u(s)\|^2_{\dot{H}^{\alpha/2}}ds=\|u_0\|^2_{L^2}.
$$
\end{teo}
In the above theorem, we impose domain restrictions and stronger smallness assumptions. These domain restrictions are due to the better behavior of the fractional laplacian in a bounded domain. The size restrictions on data are again on a lower order norm (Lipschitz) and with a rather explicit constant.

Finally we obtain the standard finite time blow up for certain initial data in some H\"{o}lder seminorm: 
\begin{prop}\label{teo5}
Fix  a constant $M>0$ and consider $\mu=0$, $\min\{\nu,\alpha\}=0$. Then, there exist $0\leq u_0\leq 1\in H^2(\Omega)$ and $T^*<\infty$ such that the corresponding solution, $u(t)$, of equation \eqref{BLstandard} has a finite time singularity in $C^\delta$ for $0<\delta\ll 1$, 
\emph{i.e.}, 
$$
\limsup_{t\rightarrow T^*}\|u(t)\|_{C^\delta}=\infty.
$$
\end{prop}
The proof of this result is obtained by a virial-type argument. However, we remark that it can also be obtained by means of pointwise arguments (see Castro \& C\'ordoba \cite{CC} for an application of these pointwise arguments to prove blow up). These virial-type arguments have been used for several transport equations even in the case of nonlocal velocities (see C\'ordoba,C\'ordoba \& Fontelos \cite{gazolaz2005formation}, Dong, Du \& Li \cite{dong2009finite}, Li \& Rodrigo \cite{LiRodrigo2} and Li, Rodrigo \& Zhang \cite{li2010exploding}). In this case, the transport term is highly nonlinear and this method fails in the case of viscosity $\nu>0,$ $0<\alpha\ll1$.

\subsection{Notation and preliminaries}\label{ssec:not}

\subsubsection{Singular integral operators}
We denote the usual Fourier transform of $u$ by $\hat{u}$. Given a function $u:\Omega\rightarrow \RR$, we write $\Lambda^\alpha u=(-\Delta)^{\alpha/2}u$ for the fractional laplacian, \emph{i.e.}
$$
\widehat{\Lambda^\alpha u}(\xi)=|\xi|^\alpha\hat{u}(\xi).
$$
This operator admits the kernel representation
\begin{equation}\label{1}
\Lambda^\alpha u(x)=\mathcal{C}_{\alpha,d}\sum_{\gamma\in\mathbb{Z}^d}\text{P.V.}\int_{\TT^d}\frac{u(x)-u(y)}{|x-y-2\pi\gamma|^{d+\alpha}}dy,
\end{equation}
if the function is periodic and
\begin{equation}\label{2}
\Lambda^\alpha u(x)=\mathcal{C}_{\alpha,d}\text{P.V.}\int_{\RR^d}\frac{u(x)-u(y)}{|x-y|^{d+\alpha}}dy,
\end{equation}
if the function is flat at infinity. Notice that we have
\begin{equation}\label{normalizingconstant}
\mathcal{C}_{\alpha,1}=\frac{\Gamma(1+\alpha)\cos((1-\alpha)\pi/2)}{\pi},
\end{equation}
where $\Gamma(\cdot)$ denotes the classical $\Gamma$ function.

\subsubsection{Functional spaces}
We write $H^s(\Omega^d)$ for the usual $L^2$-based Sobolev spaces with norm
$$
\|f\|_{H^s}^2=\|f\|_{L^2}^2+\|f\|_{\dot{H}^s}^2, \quad \|f\|_{\dot{H}^s}=\|\Lambda^s f\|_{L^2}.
$$
The fractional $L^p$-based Sobolev spaces, $W^{s,p}(\Omega^d)$, are 
$$
W^{s,p}=\left\{f\in L^p(\Omega^d), \pax^{\lfloor s\rfloor} f\in L^p(\Omega^d), \frac{|\pax^{\lfloor s\rfloor}f(x)-\pax^{\lfloor s\rfloor}f(y)|}{|x-y|^{\frac{d}{p}+(s-\lfloor s\rfloor)}}\in L^p(\Omega^d\times\Omega^d)\right\},
$$
with norm
$$
\|f\|_{W^{s,p}}^p=\|f\|_{L^p}^p+\|f\|_{\dot{W}^{s,p}}^p, 
$$
where
$$
\|f\|_{\dot{W}^{s,p}}^p=\|\pax^{\lfloor s\rfloor} f\|^p_{L^p}+\int_{\Omega^d}\int_{\Omega^d}\frac{|\pax^{\lfloor s\rfloor}f(x)-\pax^{\lfloor s\rfloor}f(y)|^p}{|x-y|^{d+(s-\lfloor s \rfloor)p}}dxdy.
$$

\subsubsection{Entropy functionals}
For a given function $u\geq0$, we define the following entropy functionals
\begin{equation}\label{LlogL}
\mathcal{F}_1[u]=\int_{\Omega^d}\left(u(x)\log(u(x))-u(x)+1\right)dx,
\end{equation}
\begin{equation}\label{(1+L)log(1+L)}
\mathcal{F}_2[u]=\int_{\Omega^d}(1+u(x))\log(1+u(x))dx.
\end{equation}
These two entropies have an associated Fisher information:
\begin{equation}\label{FisherLlogL}
\mathcal{I}^\alpha_1[u]=\int_{\Omega^d}\Lambda^\alpha u(x)\log(u(x))dx,
\end{equation}
\begin{equation}\label{Fisher(1+L)log(1+L)}
\mathcal{I}^\alpha_2[u]=\int_{\Omega^d}\Lambda^\alpha u(x)\log(1+u(x))dx.
\end{equation}
The third entropy that we are using reads
\begin{equation}\label{LlogL2}
\mathcal{F}_3[u]=\int_{\TT}u(x)\log\left(u(x)^2+M(1-u(x))^2\right)dx,
\end{equation}
with its Fisher information
\begin{equation}\label{FisherLlogL2}
\mathcal{I}^\alpha_3[u]=\int_\TT\Lambda^\alpha u\log\left(u^2+M(1-u)^2\right)dx+2(M+1)\int_{\TT}\frac{\Lambda^{\alpha}uu^2}{u^2+M(1-u)^2}dx.
\end{equation}

\subsubsection{Notation} Recall that we denote the mean of a function by 
$$
\langle u\rangle=\frac{1}{2\pi}\int_\TT u(y)dy.
$$
Let us introduce $f$ and $a$ as follows
\begin{equation}\label{eqf}
f(u(x,t))=\frac{u(x,t)^2}{u(x,t)^2+M(1-u(x,t))^2},
\end{equation}
and
\begin{equation}\label{eqa}
a(u(x,t))=\frac{df}{du}=\frac{2u(x,t)M(1-u(x,t))}{(u(x,t)^2+M(1-u(x,t))^2)^2}.
\end{equation}
Finally, let us introduce the notation for the mollifiers. For $\epsilon>0$, we write $\mathcal{J}_\epsilon$ for the heat kernel at time $t=\epsilon$ and define 
$$
\jeps f=\mathcal{J}_\epsilon f.
$$
\subsubsection{Weak solutions to \eqref{BL} and their local existence}

We start this section with 
\begin{defi}
Let $\mu, \nu \ge 0$ and $0<T<\infty$ be a fixed positive parameter. The function 
$$
u\in L^\infty([0,T)\times \Omega)
$$
 is a (very weak) solution of \eqref{BL} if
$$
\int_0^T\int_\Omega [\pat\phi-\nu\Lambda^\alpha\phi + \mu \Lambda^\beta \pat \phi]u +\pax\phi\left[\frac{u^2}{u^2+M(1-u)^2}\right] = \int_\Omega [\mu  \Lambda^\beta \pat \phi + \phi]_{|{t=0}} \, u_0,
$$
for every test function $\phi(x,t)\in C^\infty_c([0,T)\times\Omega)$. 

If a solution $u$ verifies the previous definition for every $0<T<\infty$, it is called a global solution.
\end{defi}

\begin{lem}\label{lem1}
Let $u_0\in H^s$, $s\geq2$ be the initial data for \eqref{BL} with $\mu, \nu \geq0$ and $0\leq \alpha,\beta\leq 2$. Then there exist $T(u_0,M)$ and the unique solution 
$$
u\in C([0,T(u_0,M)],H^s(\Omega))
$$
to \eqref{BL}. The maximal time of existence $T(u_0,M)$ is characterized by
\[\limsup_{t \to T(u_0,M)} \|u (t)\|_{\dot{W}^{1,\infty}(\Omega)} =  \infty.\]
\end{lem}
\begin{proof}
Using the same ideas as in \cite[Theorem 3.1]{HongWuYuan}, we can construct solutions to the regularized problem
$$
\pat u_{\epsilon,\delta}+\pax \left[\frac{u_{\epsilon,\delta}^2}{u_{\epsilon,\delta}^2+M(1-u_{\epsilon,\delta})^2}\right]=-\nu\Lambda^\alpha u_{\epsilon,\delta}-\mu\Lambda^\beta \pat u_{\epsilon,\delta}+\epsilon \pax^2u_{\epsilon,\delta}+\delta\pax^2\pat u_{\epsilon,\delta},
$$
with initial data $u_{\epsilon,\delta}(0)=u_0$. Standard energy estimates give us uniform bounds. Then we can pass to the limits $\epsilon,\delta\rightarrow0$. The proof of the continuation criteria can be obtained by energy methods.
\end{proof}

\section{The entropy inequalities}\label{sec2}
In this section we provide the proof of three entropy inequalities that, in our opinion, may be of independent interest.
\begin{prop}\label{prop1}
Let $u$ be a given function and $0<\alpha<2$, $0<\epsilon<\alpha/2$ be two fixed constants. Then
\begin{equation}\label{entropyp1}
\|u\|_{\dot{W}^{\alpha/2-\epsilon,1}}^2\leq C(\alpha,d,\epsilon)\|u\|_{L^1}\mathcal{I}^{\alpha}_i[u],
\end{equation}
\begin{equation}\label{entropypinfty}
\|u\|_{\dot{H}^{\alpha/2}}^2\leq \frac{4}{\mathcal{C}_{\alpha,d}}\|u\|_{L^\infty}\mathcal{I}^{\alpha}_i[u],
\end{equation}
provided that the right hand sides are meaningful.
\end{prop}
\begin{proof}
Let us fix $i=1$. First we symmetrize
\begin{eqnarray*}
\mathcal{I}^\alpha_1[u]&=&\mathcal{C}_{\alpha,d}\sum_{\gamma\in\mathbb{Z}^d}\int_{\TT^d}
\text{P.V.}\int_{\TT^d}\frac{u(x)-u(y)}{|x-y-2\pi\gamma|^{d+\alpha}}\log(u(x))dydx\\
&=&-\mathcal{C}_{\alpha,d}\sum_{\gamma\in\mathbb{Z}^d}\int_{\TT^d}
\text{P.V.}\int_{\TT^d}\frac{u(x)-u(y)}{|x-y-2\pi\gamma|^{d+\alpha}}\log(u(y))dydx\\
&=&\frac{\mathcal{C}_{\alpha,d}}{2}\sum_{\gamma\in\mathbb{Z}^d}\int_{\TT^d}
\text{P.V.}\int_{\TT^d}\frac{u(x)-u(y)}{|x-y-2\pi\gamma|^{d+\alpha}}\log\left(\frac{u(x)}{u(y)}\right)dydx\\
&\geq&0.
\end{eqnarray*}
Furthermore, since $(a-b) \log \left(\frac{a}{b} \right)\geq0$, every term in the series is positive, \emph{i.e.} for every $\gamma\in\mathbb{Z}^d$, we have
$$
\int_{\TT^d}
\text{P.V.}\int_{\TT^d}\frac{u(x)-u(y)}{|x-y-2\pi\gamma|^{d+\alpha}}\log\left(\frac{u(x)}{u(y)}\right)dydx\geq0.
$$
In particular
\begin{equation}\label{lower}
\mathcal{I}^\alpha_1[u]\geq \frac{\mathcal{C}_{\alpha,d}}{2}\int_{\TT^d}
\text{P.V.}\int_{\TT^d}\frac{u(x)-u(y)}{|x-y|^{d+\alpha}}\log\left(\frac{u(x)}{u(y)}\right)dydx
\end{equation}
Let us consider first the case $0\leq u\in L^1$. We have
\begin{eqnarray*}
\|u\|_{\dot{W}^{\alpha/2-\epsilon,1}}&=&\int_{\TT^d}\int_{\TT^d}\frac{|u(x)-u(y)|}{|x-y|^{d+\frac{\alpha}{2}-\epsilon}}dxdy\\
&=&\int_{\TT^d}\int_{\TT^d}\int_0^1 \bigg[\frac{|u(x)-u(y)|}{|x-y|^{d+\frac{\alpha}{2}-\epsilon}}\frac{|x-y|^{-\frac{d}{2}+\epsilon}}{|x-y|^{-\frac{d}{2}+\epsilon}}\\
&&\times \frac{|su(x)+(1-s)u(y)|^{1/2}}{|su(x)+(1-s)u(y)|^{1/2}}\bigg]dsdxdy\\
&\leq& I_1^{0.5}I_2^{0.5},
\end{eqnarray*}
with
$$
I_1=\int_{\TT^d}\int_{\TT^d}\int_0^1 \frac{|u(x)-u(y)|^2}{|x-y|^{d+\alpha}}\frac{1}{|su(x)+(1-s)u(y)|}dsdxdy,
$$
$$
I_2=\int_{\TT^d}\int_{\TT^d}\int_0^1 \frac{|su(x)+(1-s)u(y)|}{|x-y|^{d-2\epsilon}}dsdxdy.
$$
This latter integral is similar to the Riesz potential. Due to the positivity of $u$, we have
\begin{eqnarray*}
I_2&=&\int_{\TT^d}\int_{\TT^d}\int_0^1 \frac{su(x)+(1-s)u(y)}{|x-y|^{d-2\epsilon}}dsdxdy\\
&=&\mathcal{C}(\epsilon)\|u\|_{L^1},
\end{eqnarray*}
with
$$
\mathcal{C}(\epsilon)=\int_{\TT^d}\frac{1}{|y|^{d-2\epsilon}}dy.
$$
We have
$$
I_1=\int_{\TT^d}
\text{P.V.}\int_{\TT^d}\frac{u(x)-u(y)}{|x-y|^{d+\alpha}}\log\left(\frac{u(x)}{u(y)}\right)dydx.
$$
Consequently, we get
$$
\|u\|_{\dot{W}^{\alpha/2-\epsilon,1}}^2\leq 2\frac{\mathcal{C}(\epsilon)}{\mathcal{C}_{\alpha,d}}\|u\|_{L^1}\mathcal{I}^{\alpha}_1[u].
$$
The case $0\leq u\in L^\infty$ was first proved by Bae \& Granero-Belinch\'on \cite{bae2015global}. For the sake of completeness, we include here a sketch of the proof. Using \eqref{lower}, we have
$$
\mathcal{I}^\alpha_1[u]\geq \frac{\mathcal{C}_{\alpha,d}}{2}\frac{1}{2\|u\|_{L^\infty}}\int_{\TT^d}\int_{\TT^d}\frac{|u(x)-u(y)|^2}{|x-y|^{d+\alpha}}dxdy,
$$
so
$$
\|u\|_{\dot{H}^{\alpha/2}}^2\leq \frac{4}{\mathcal{C}_{\alpha,d}}\|u\|_{L^\infty}\mathcal{I}^{\alpha}_1[u].
$$
The proof for the case $i=2$ is similar.
\end{proof}

\section{Proof of Proposition \ref{teo1}: Weak solutions}\label{sec3}

We prove the result for $\Omega=\TT$, but the same proof can be adapted to deal with $\Omega=\RR$. We consider the regularized problems
\begin{equation}\label{BLreg}
\pat u_\epsilon+\pax \left[\frac{u_\epsilon^2}{u_\epsilon^2+M(1+\epsilon-u_\epsilon)^2}\right]=-\nu\Lambda^\alpha u_\epsilon+\epsilon\pax^2 u_\epsilon,\, x\in\TT,\,t>0, \end{equation}
with the regularized initial data
$$
u_\epsilon(x,0)=\jeps u_0(x)+\epsilon,
$$
and $\epsilon \le 1/2$. These approximate problems have global classical solution due to the Theorem 3.1 in \cite{HongWuYuan}. Consequently, we focus on obtaining the appropriate $\epsilon$-uniform bounds. By assumption $0\leq u(x,0)\leq 1$. We apply the same technique as C\'ordoba \& C\'ordoba \cite{cor2} (see also \cite{CC,c-g09,ccgs-10,bae2015global,AGM,G,GH,GNO,BG,GO} for more details and application to other partial differential equations), \emph{i.e.} we track
$$
\|u_\epsilon(t)\|_{L^\infty}=u_\epsilon(\overline{x}_t,t)=\mathcal{M}_\epsilon(t)
$$
and 
$$
\min_{x \in\Omega} u_\epsilon(x, t)=u_\epsilon(\underline{x}_t,t)=\mathfrak{M}_\epsilon(t).
$$
Due to smoothness of $u_\epsilon$, we have that $\mathcal{M}_\epsilon(t), \mathfrak{M}_\epsilon(t)$ are Lipschitz, and consequently almost everywhere differentiable. Hence, using 
$$
\mathcal{M}_\epsilon'(t)=\frac{d}{dt}\|u(t)\|_{L^\infty}=\frac{d}{dt} (u(\overline{x}_t,t)), \qquad \mathfrak{M}_\epsilon'(t) =\frac{d}{dt} (u_\epsilon(\underline{x}_t,t)
$$

together with the kernel expression for $\Lambda^\alpha$, we have the $\epsilon$-uniform bounds
$$
0\leq\epsilon\leq u_\epsilon(x,t)\leq 1.5.
$$
By space integration of \eqref{BLreg} we get
$$
\|u_\epsilon(t)\|_{L^1}=\|u_\epsilon(0)\|_{L^1}\leq \|u_0\|_{L^1}+\pi.
$$
We compute
\begin{eqnarray*}
\frac{d}{dt}\mathcal{F}_1[u_\epsilon]&=&\int_\TT \pat u_\epsilon\log(u_\epsilon)dx\\
&=&\int_\TT\frac{u_\epsilon\pax u_\epsilon}{u_\epsilon^2+M(1+\epsilon-u_\epsilon)^2}dx-\nu\mathcal{I}_1^\alpha[u_\epsilon]
-\epsilon\mathcal{I}_1^2[u_\epsilon].
\end{eqnarray*}
We have
$$
\pax\log\left(u_\epsilon^2+M(1+\epsilon-u_\epsilon)^2\right)=\frac{2u_\epsilon\pax u_\epsilon(1+M)-2M(1+\epsilon)\pax u_\epsilon}{u_\epsilon^2+M(1+\epsilon-u_\epsilon)^2},
$$
and, as a consequence
\begin{eqnarray*}
\int_\TT\left[\frac{u_\epsilon\pax u_\epsilon}{u_\epsilon^2+M(1+\epsilon-u_\epsilon)^2}\right]dx&=&\frac{M(1+\epsilon)}{M+1}\int_\TT\frac{\pax u_\epsilon}{u_\epsilon^2+M(1+\epsilon-u_\epsilon)^2}dx\\
&=&\frac{M(1+\epsilon)}{M+1}\int_\TT\frac{\pax u_\epsilon}{u_\epsilon^2}\frac{1}{1+M\left(\frac{1+\epsilon-u_\epsilon}{u_\epsilon}\right)^2}dx\\
&=&\frac{-M}{M+1}\int_\TT\frac{\pax\left[\frac{1+\epsilon-u_\epsilon}{u_\epsilon}\right]}{1+M\left(\frac{1+\epsilon-u_\epsilon}{u_\epsilon}\right)^2}dx\\
&=&\frac{-\sqrt{M}}{M+1}\int_\TT \pax\left[ \arctan \left( \sqrt{M} \left(\frac{1+\epsilon-u_\epsilon}{u_\epsilon}\right) \right) \right] dx\\
&=&0.
\end{eqnarray*}
Thus, we conclude
$$
\mathcal{F}_1[u_\epsilon(t)]+\nu\int_0^t\mathcal{I}_1^\alpha[u_\epsilon(s)]ds
+\epsilon\int_0^t\mathcal{I}_1^2[u_\epsilon(s)]ds=\mathcal{F}_1[u_\epsilon(0)].
$$
As 
$$
\mathcal{F}_1[u_\epsilon(0)]\leq C,
$$
we get a $\epsilon$-uniform estimate. Now we apply \eqref{entropypinfty} from Proposition \ref{prop1} and $u_\epsilon(x,t)\leq 1.5$ to get
$$
\int_{0}^t\|u_\epsilon(s)\|_{H^{\alpha/2}}^2ds\leq C.
$$
Consequently, we have $\epsilon$-uniform bounds
$$
u_\epsilon(t)\in L^\infty(0,\infty;L^\infty)\cap L^2(0,\infty;H^{\alpha/2}).
$$
and the first entropy inequality \eqref{eq:enA}. For the second entropy inequality \eqref{eq:enB}, we compute
\begin{eqnarray*}
\frac{d}{dt}\mathcal{F}_3[u]&=&\int_{\TT}\pat u\log\left(u^2+M(1-u)^2\right)dx\\
&&+\int_{\TT}2\frac{u^2\pat u+uM(1-u)(-\pat u)}{u^2+M(1-u)^2}dx\\
&=&I_1+I_2
\end{eqnarray*}
We are going to handle both integrals separately. We have
\begin{eqnarray*}
I_1&=&\int_\TT\frac{u^2}{u^2+M(1-u)^2}\frac{2u\pax u+2M(1-u)(-\pax u)}{u^2+M(1-u)^2}dx\\
&&-\nu\int_\TT\Lambda^\alpha u\log\left(u^2+M(1-u)^2\right)dx\\
&=& \int_\TT\frac{2u\pax u}{u^2+M(1-u)^2}dx-\nu\int_\TT\Lambda^\alpha u\log\left(u^2+M(1-u)^2\right)dx\\
&=&-\nu\int_\TT\Lambda^\alpha u\log\left(u^2+M(1-u)^2\right)dx,
\end{eqnarray*}
where in the second term we have used
$$
\frac{2u\pax u+2M(1-u)(-\pax u)}{(u^2+M(1-u)^2)^2}=-\pax\left(\frac{1}{u^2+M(1-u)^2}\right).
$$
The second integral reads
\begin{eqnarray*}
I_2&=&-2\nu(M+1)\int_{\TT}\frac{\Lambda^{\alpha}uu^2}{u^2+M(1-u)^2}dx-M\int_\TT\frac{2u\pat u}{u^2+M(1-u)^2}dx.
\end{eqnarray*}
Now we compute
$$
\frac{d}{dt}\int_\TT\log\left(u^2+M(1-u)^2\right)dx=2\int_\TT\frac{(1+M)u\pat u-M\pat u}{u^2+M(1-u)^2}dx,
$$
so
\begin{multline*}
\int_\TT\frac{2 u\pat u}{u^2+M(1-u)^2}dx=\frac{1}{1+M}\left[\frac{d}{dt}\int_\TT\log\left(u^2+M(1-u)^2\right)dx\right.\\
\left.-2\sqrt{M}\frac{d}{dt}\int_\TT\arctan\left(\sqrt{M}\left(\frac{1}{u}-1\right)\right)dx\right]
\end{multline*}
Collecting all these computations we get
\begin{multline*}
\mathcal{F}_3[u(t)]+\nu\int_0^t\mathcal{I}^\alpha_3[u(s)]ds+\frac{M}{1+M}\left[\int_\TT\log\left(u(t)^2+M(1-u(t))^2\right)dx\right.\\
\left.-2\sqrt{M}\int_\TT\arctan\left(\sqrt{M}\left(\frac{1}{u(t)}-1\right)\right)dx\right]\\
=\mathcal{F}_3[u_0]+\frac{M}{1+M}\left[\int_\TT\log\left(u_0^2+M(1-u_0)^2\right)dx\right.\\
\left.-2\sqrt{M}\int_\TT\arctan\left(\sqrt{M}\left(\frac{1}{u_0}-1\right)\right)dx\right].
\end{multline*}

\section{Proof of Proposition \ref{teo1b}: Decay estimates}

Let us prove first the periodic case. The $L^1$ norm is preserved. Consequently, the mean propagates. We again apply the technique of tracking $\mathcal{M}(t)$ and $\mathfrak{M}(t)$. Recall 
$$
\|u(t)\|_{L^\infty}=u(\overline{x}_t,t)=\mathcal{M}(t).
$$
The smoothness needed to proceed with $\mathcal{M}(t)'$ is, for this proposition, an assumption.

Since $x, y \in\TT$,  $|x-y|^{1+\alpha}\leq (2\pi)^{1+\alpha}$. Hence, using \eqref{1} and \eqref{normalizingconstant}, we get
$$
\Lambda^\alpha u(\overline{x}_t)\geq \frac{2\Gamma(1+\alpha)\cos((1-\alpha)\pi/2)}{(2\pi)^{1+\alpha}}(u(\overline{x}_t,t)-\langle u_0\rangle).
$$
Consequently,
\begin{align*}
\mathcal{M}'(t)& =\frac{d}{dt}\left(u(\overline{x}_t,t)-\langle u_0\rangle\right)\\
& =-\Lambda^\alpha u(\overline{x}_t)\\
&\leq -\frac{2\Gamma(1+\alpha)\cos((1-\alpha)\pi/2)}{(2\pi)^{1+\alpha}}(u(\overline{x}_t,t)-\langle u_0\rangle).
\end{align*}
Integrating this ODI, we have
$$
\|u(t)\|_{L^\infty(\TT)}\leq \langle u_0\rangle+(\|u_0\|_{L^\infty}-\langle u_0\rangle)e^{-\frac{2\Gamma(1+\alpha)\cos((1-\alpha)\pi/2)}{(2\pi)^{1+\alpha}}t}.
$$
Let us turn our attention to the flat at infinity case. Again the $L^1$ norm propagates. We take a positive number $r>0$  (that will be specified below) and define 
$$
\mathcal{U}_1=\{\eta\in B(0,r) \;s.t.\; u(\overline{x}_t)-u(\overline{x}_t-\eta)>u(\overline{x}_t,t)/2 \},
$$
and $\mathcal{U}_2=B(0,r)-\mathcal{U}_1$. We have
$$
\|u_0\|_{L^1}=\int_\RR|u(\overline{x}_t-y)| dy\geq \int_{\mathcal{U}_2}u(\overline{x}_t-y)dy\geq\frac{ u(\overline{x}_t)}{2}|\mathcal{U}_2|,
$$
so (for $u(\overline{x}_t) > 0$),
\begin{equation}\label{eqappaux}
-\frac{2\|u_0\|_{L^1}}{u(\overline{x}_t)}\leq -|\mathcal{U}_2|.
\end{equation}
\begin{eqnarray*}
\Lambda^\alpha u(\overline{x}_t)&=&c(\alpha)\text{P.V.}\int_\RR\frac{u(\overline{x}_t)-u(\overline{x}_t-y)}{|y|^{1+\alpha}}dy\\
&\geq& c(\alpha)\text{P.V.}\int_{\mathcal{U}_1}\frac{u(\overline{x}_t)-u(\overline{x}_t-y)}{|y|^{1+\alpha}}dy\\
&\geq& c(\alpha)\frac{u(\overline{x}_t)}{2r^{1+\alpha}}\left(2r-\frac{2\|u_0\|_{L^1}}{u(\overline{x}_t)}\right).
\end{eqnarray*}
We choose now 
$$
r=\frac{2\|u_0\|_{L^1}}{u(\overline{x}_t)},
$$
thus, recalling \eqref{normalizingconstant}, we have that
$$
\Lambda^\alpha u(\overline{x}_t)\geq \frac{\mathcal{C}_{\alpha,1}}{2^{1+ \alpha}}\frac{u(\overline{x}_t)^{1+\alpha}}{\|u_0\|^{\alpha}_{L^1}}
$$
for both $u(\overline{x}_t) > 0$ and $u(\overline{x}_t) = 0$.
With the same argument as in the periodic case, we get
$$
\frac{d}{dt}\|u(t)\|_{L^\infty}\leq -\frac{\mathcal{C}_{\alpha,1}}{2^{1+ \alpha}}\frac{\|u(t)\|_{L^\infty}^{1+\alpha}}{\|u_0\|^{\alpha}_{L^1}},
$$
thus, using explicit value of $\mathcal{C}_{\alpha,1}$ we arrive at
\begin{align*}
\|u(t)\|_{L^\infty} & \leq  \frac{\|u_0\|_{L^\infty}}{\left(1+\alpha \frac{\Gamma(1+\alpha)\cos((1-\alpha)\pi/2)}{2^{1+\alpha} \pi} \frac{\|u_0\|_{L^\infty}^{\alpha}}{ \|u_0\|^{\alpha}_{L^1} }t \right)^{\frac{1}{\alpha}}}.
\end{align*}

\section{Proof of Theorem \ref{teo2}: Global solutions for $\max\{\alpha,\beta\}>1$}
Equipped with the Lemma \ref{lem1} and its proof, we can focus on the appropriate energy estimates that ensures global existence (rigorously, we should do this on the level of the regularized problem)

Notice that we also have a global bound
\begin{equation}\label{L^2}
\|u(t)\|_{L^2}^2+\mu\|u(t)\|^2_{\dot{H}^{\beta/2}}+ 2\nu\int_0^t\|u(s)\|_{\dot{H}^{\alpha/2}}^2ds\leq\|u_0\|_{L^2}+\mu\|u_0\|^2_{\dot{H}^{\beta/2}}.
\end{equation}
We split the proof in three parts: the first one is devoted to the proof of the purely parabolic case $\mu=0$. Then, in step 2 and 3 we consider the cases $\mu>0 $ $\beta<1<\alpha$ and $\mu>0 $ $\alpha<1<\beta$, respectively.

\textbf{Step 1: Case $\mu=0$.} We perform the estimates for $s=1$, the case $s>1$ being analogous. Testing \eqref{BL} against $\Lambda u$ and using the self-adjointness, we have that
\begin{equation}\label{4}
\frac{1}{2}\frac{d}{dt}\|u\|^2_{\dot{H}^{0.5}}=I_1+I_2,
\end{equation}
with
\begin{align}
I_1&=-\int_\Omega \left(\frac{2u}{u^2+M(1-u)^2}-\frac{u^2(2u-2M(1-u))}{\left(u^2+M(1-u)^2\right)^2}\right)\pax u\Lambda udx \label{5}\\
I_2&=-\nu\|u\|_{\dot{H}^{\frac{1+\alpha}{2}}}^2.\label{6}
\end{align}
Due to H\"{o}lder's inequality
\begin{equation}\label{7}
I_1\leq \|u\|_{\dot{H}^1}^2\left\|\frac{2u}{u^2+M(1-u)^2}-\frac{u^2(2u-2M(1-u))}{\left(u^2+M(1-u)^2\right)^2}\right\|_{L^\infty}.
\end{equation}
Using interpolation
$$
\|f\|_{\dot H^{1}}\leq c\|f\|_{\dot H^{0.5}}^{1-\frac{1}{\alpha}}\|f\|_{\dot H^{\frac{1+\alpha}{2}}}^{\frac{1}{\alpha}},
$$
and  \eqref{boundM} we have
$$
I_1\leq \|u\|_{\dot H^{0.5}}^{2-\frac{2}{\alpha}}\|u\|_{\dot{H}^{\frac{1+\alpha}{2}}}^{\frac{2}{\alpha}}c(u_0,M)\leq \|u\|_{H^{0.5}}^{2}c(u_0,M,\nu) + \nu \|u\|_{\dot{H}^{\frac{1+\alpha}{2}}}^{2}.
$$
We conclude
\begin{equation}\label{3}
\|u(t)\|_{H^{0.5}}\leq \|u_0\|_{H^{0.5}}e^{c(u_0,M,\nu)t}.
\end{equation}
Testing \eqref{BL} against $-\pax^2 u$ and integrating, we have
\begin{equation}\label{8}
\frac{1}{2}\frac{d}{dt}\|u\|^2_{\dot{H}^1}=I_3+I_4,
\end{equation}
with
\begin{align}
I_3&=-\int_\Omega \left(\frac{2u}{u^2+M(1-u)^2}-\frac{u^2(2u-2M(1-u))}{\left(u^2+M(1-u)^2\right)^2}\right)\pax u\pax^2udx\label{9}\\
I_4&=-\nu\|u\|_{\dot{H}^{1+\frac{\alpha}{2}}}^2.\label{10}
\end{align}
We have the bounds
\begin{equation}\label{bds}
\begin{aligned}
\left\|\left(\frac{2u}{u^2+M(1-u)^2}-\frac{u^2(2u-2M(1-u))}{\left(u^2+M(1-u)^2\right)^2}\right)\right\|_{L^2} & \leq c(u_0,M)\|u\|_{L^2},\\
\left\|\left(\frac{2u}{u^2+M(1-u)^2}-\frac{u^2(2u-2M(1-u))}{\left(u^2+M(1-u)^2\right)^2}\right)\right\|_{H^{1}} & \leq c(u_0,M)\|u\|_{H^1},
\end{aligned}
\end{equation}
so, using interpolation, we have
$$
\left\|\left(\frac{2u}{u^2+M(1-u)^2}-\frac{u^2(2u-2M(1-u))}{\left(u^2+M(1-u)^2\right)^2}\right)\right\|_{H^{1-\alpha/2}}\leq c(u_0,M)\|u\|_{H^{1-\alpha/2}}.
$$
We use the duality pairing $H^{1-\alpha/2}-H^{\alpha/2-1}$ together with Moser's inequality and embeddings to obtain
\begin{align*}
I_3&\leq \left\|\left(\frac{2u}{u^2+M(1-u)^2}-\frac{u^2(2u-2M(1-u))}{\left(u^2+M(1-u)^2\right)^2}\right)\pax u\right\|_{H^{1-\alpha/2}}\|\pax^2u\|_{H^{\alpha/2-1}}\\
&\leq c\left\|\left(\frac{2u}{u^2+M(1-u)^2}-\frac{u^2(2u-2M(1-u))}{\left(u^2+M(1-u)^2\right)^2}\right)\right\|_{L^\infty}\|\pax u\|_{H^{1-\alpha/2}}\|u\|_{H^{1+\alpha/2}}\\
&\quad+c\left\|\left(\frac{2u}{u^2+M(1-u)^2}-\frac{u^2(2u-2M(1-u))}{\left(u^2+M(1-u)^2\right)^2}\right)\right\|_{H^{1-\alpha/2}}\|\pax u\|_{L^\infty}\|u\|_{H^{1+\alpha/2}}\\
&\leq c(u_0,M,\nu)e^{c(u_0,M,\nu)t} +\frac{\nu}{2}\|u\|_{H^{1+\alpha/2}}^2.
\end{align*}
Notice that the above estimate can be also derived using the Kato-Ponce inequality. As a consequence,
$$
\|u(t)\|_{H^1}^2+\frac{\nu}{2}\int_0^t\|u(s)\|_{H^{1+\alpha/2}}^2ds\leq c(u_0,M,\nu)e^{c(u_0,M,\nu)t}.
$$

\textbf{Step 2: Case $\mu>0$, $\beta>1$ (and $0\leq \alpha\leq 2$).} We consider the case $s=\frac{1+\beta}{2}$, the other cases ($s>\frac{1+\beta}{2}$) being analogous. Using \eqref{L^2} together with $\beta>1$, we have
$$
\|u(t)\|_{L^\infty}\leq C\|u(t)\|_{H^{\beta/2}}\leq C(\beta,\mu,M,u_0).
$$
Now we test \eqref{BL} against $\Lambda u$. Due to \eqref{5}-\eqref{7} with  \eqref{boundM}, we obtain
$$
\frac{1}{2}\frac{d}{dt}\left(\|u(t)\|_{\dot{H}^{0.5}}^2+\mu\|u(t)\|_{\dot{H}^{\frac{1+\beta}{2}}}^2\right)+ 2\nu\|u(t)\|^2_{\dot{H}^{\frac{1+\alpha}{2}}}\leq C(\beta,\mu,M,u_0)\|u\|_{\dot{H}^1}^2.
$$
Integrating, we obtain
$$
\|u(t)\|_{\dot{H}^{0.5}}^2+\mu\|u(t)\|_{\dot{H}^{\frac{1+\beta}{2}}}^2+\int_0^t 2\nu\|u(s)\|^2_{\dot{H}^{\frac{1+\alpha}{2}}}ds\leq C \|u_0\|_{H^{\frac{1+\beta}{2}}}^2e^{C(\beta,\mu,M,u_0)t}.
$$
Testing against $-\pax^2 u$ we can perform energy estimates as in Step 1. We obtain 
$$
\frac{1}{2}\frac{d}{dt}\left(\|u(t)\|_{\dot{H}^{1}}^2+\mu\|u(t)\|_{\dot{H}^{1+\frac{\beta}{2}}}^2\right)+\nu\|u(t)\|^2_{\dot{H}^{1+\frac{\alpha}{2}}}\leq C(\beta,\mu,M,u_0)\|\pax u(t)\|^3_{L^3}.
$$
We use the interpolation inequality
$$
\|f\|_{L^4}\leq C\|f\|^{0.5}_{H^{\beta/2}}\|f\|^{0.5}_{L^{2}},
$$
so
\begin{align*}
C\|\pax u(t)\|^3_{L^3} & \leq C\|\pax u\|^2_{L^4}\|\pax u\|_{L^2}\\
&\leq C\|u\|_{H^{1+\beta/2}}\|u\|_{H^{1}}^2\\
&\leq \frac{\mu}{4}\|u\|_{H^{1+\beta/2}}^2+C(\mu,u_0)e^{C(\beta,\mu,M,u_0)t}.
\end{align*}
Now we can use Gronwall's inequality to obtain
$$
\|u(t)\|_{\dot{H}^{1}}^2+\frac{\mu}{2}\|u(t)\|_{\dot{H}^{1+\frac{\beta}{2}}}^2+ 2\nu\int_0^t\|u(s)\|^2_{\dot{H}^{1+\frac{\alpha}{2}}}ds\leq C(\beta,\mu,M,u_0)e^{C(\beta,\mu,M,u_0)t}.
$$
\textbf{Step 3: Case $\mu>0$, $0<\beta<1$ and $\nu >0$,  $1<\alpha \le 2$}. As before, we test \eqref{BL} against $\Lambda u$. Using \eqref{boundM}, we obtain
$$
\frac{1}{2}\frac{d}{dt}\left(\|u(t)\|_{\dot{H}^{0.5}}^2+\mu\|u(t)\|_{\dot{H}^{\frac{1+\beta}{2}}}^2\right)+\nu\|u(t)\|^2_{\dot{H}^{\frac{1+\alpha}{2}}}\leq C(M)\|u\|_{\dot{H}^1}^{2}.
$$
Now we use the interpolation 
$$
\|u(t)\|_{\dot{H}^{1}}\leq C\|u_0\|_{H^{\beta/2}}^{\frac{\alpha-1}{1+\alpha-\beta}}\|u(t)\|_{\dot{H}^{\frac{1+\alpha}{2}}}^{\frac{2-\beta}{1+\alpha-\beta}}.
$$
For $\alpha>1$ we have that
$$
\frac{2-\beta}{1+\alpha-\beta}<1,
$$
hence
$$
\|u(t)\|_{\dot{H}^{0.5}}^2+\mu\|u(t)\|_{\dot{H}^{\frac{1+\beta}{2}}}^2+2 \nu\int_0^t\|u(s)\|^2_{\dot{H}^{\frac{1+\alpha}{2}}}ds\leq e^{C(M)t}\|u_0\|^2_{\dot{H}^{\frac{1+\beta}{2}}}.
$$
Now, we test against $-\pax^2 u$. We can conclude as in Step 1. We obtain
$$
\|u(t)\|_{\dot{H}^{1}}^2+\mu\|u(t)\|_{\dot{H}^{1+\frac{\beta}{2}}}^2+ \nu \int_0^t\|u(s)\|^2_{\dot{H}^{1+\frac{\alpha}{2}}}ds\leq e^{C(M)t}\|u_0\|^2_{\dot{H}^{1+\frac{\beta}{2}}}.
$$


\section{Proof of Theorem \ref{teo3}: Global solution for $\max\{\alpha,\beta\}=1$}

\textbf{Step 1: Case $\nu>0$,  $\alpha=1$, $\mu=0$.} 

We do the case $s=1$,  the other cases being analogous. Testing \eqref{BL} against $\Lambda u$ and using the self-adjointness, we have equations \eqref{4},\eqref{5}, \eqref{6} and \eqref{7}. Notice that, under the hypothesis
$$
\|u_0\|_{L^\infty}\leq \gamma.
$$
Since $\|u(t)\|_{L^\infty} \le \|u_0\|_{L^\infty}$, we have that
\begin{multline*}
\left\|\frac{2u}{u^2+M(1-u)^2}-\frac{u^2(2u-2M(1-u))}{\left(u^2+M(1-u)^2\right)^2}\right\|_{L^\infty}\\
\leq \frac{2\gamma(M+1)}{M}+\frac{2\gamma^2(\gamma+M)(M+1)^2}{M^2}.
\end{multline*}
Using that
$$
\frac{2\gamma(M+1)}{M}+\frac{2\gamma^2(\gamma+M)(M+1)^2}{M^2}<\nu,
$$
we conclude
\begin{equation}\label{11}
\|u(t)\|_{\dot H^{0.5}}+\delta\int_0^t\|u(s)\|^2_{\dot{H}^1}ds\leq \|u_0\|_{\dot H^{0.5}},
\end{equation}
for a small enough $0<\delta$. Notice that this $\delta$ only depends on $M,u_0$ and $\nu$. 

Next, testing \eqref{BL} against $-\pax^2u$ and integrating by parts, we get \eqref{8}, \eqref{9} and \eqref{10}. If we integrate by parts in \eqref{9}, we get

\begin{equation}\label{eq:i3}
I_3\leq c(u_0,M) \| u\|_{H^1}  \|\pax u\|_{L^4}^2\leq c(u_0,M,\nu)\|u\|_{\dot{H}^1}^{4}+\frac{\nu}{2}\|u\|_{\dot {H}^{1.5}}^2 +c.
\end{equation}
The first inequality above uses also \eqref{bds} and the second the interpolation
\[
\|f\|_{L^4}\leq C\|f\|^{0.5}_{H^{0.5}}\|f\|^{0.5}_{L^{2}}\]

 Using \eqref{eq:i3} in \eqref{8}, we obtain
$$
\frac{d}{dt}\|u (t) \|^2_{\dot H^1}+\nu \|u (t) \|_{\dot{H}^{1.5}}^2 \leq c(u_0,M,\nu)\|u\|_{\dot{H}^1}^{4} + c,
$$
and, due to Gronwall's inequality together wit \eqref{11}, we obtain
$$
\|u\|^2_{H^1}+\nu\int_0^t\|u(s)\|_{{H}^{1.5}}^2ds\leq c(u_0,M,\nu)e^{c(u_0,M,\nu)t}.
$$
This ends the proof of case (ii) of our thesis.

\textbf{Step 2: Case $\nu>0, \alpha=1, \mu>0, \beta<1$.} In this case we can not use the pointwise methods, so we cannot get immediately $\|u(t)\|_{L^\infty} \le \|u_0\|_{L^\infty}$. Estimate \eqref{L^2} implies a global bound in $H^{\beta/2}$, but this bound is too weak to give us a pointwise estimate for $u$. However, as
$$
\|u_0\|_{L^\infty}\leq C_S\|u_0\|_{H^{\frac{1+\beta}{2}}}\leq \gamma,
$$
testing \eqref{BL} against $\Lambda u$ and using the definition of $\gamma$ and $\gamma^*$, we have, as in step 1,
$$
\frac{1}{2}\frac{d}{dt}\left(\|u(t)\|_{\dot{H}^{0.5}}^2+\mu\|u(t)\|^2_{\dot{H}^{\frac{1+\beta}{2}}}\right)+\delta \|u(t)\|_{\dot{H}^{1}}^2\leq 0.
$$
As a consequence, we obtain the global bound
$$
\|u(t)\|_{\dot{H}^{0.5}}^2+\mu\|u(t)\|^2_{\dot{H}^{\frac{1+\beta}{2}}}+ \delta \int_0^t\|u(s)\|_{\dot{H}^{1}}^2ds\leq \|u_0\|_{\dot{H}^{0.5}}^2+\mu\|u_0\|^2_{\dot{H}^{\frac{1+\beta}{2}}}.
$$
Now we test against $-\pax^2u$ and we conclude as in Step 1. Case (iii) is proved.

\textbf{Step 3: Case $\nu\geq0, \mu>0$, $\beta=1$.} In this case, \eqref{L^2} implies a global bound in $H^{0.5}$. Then, testing \eqref{BL} against $\Lambda u$ and using \eqref{boundM}, we have
$$
\frac{1}{2}\frac{d}{dt}\left(\|u(t)\|_{\dot{H}^{0.5}}^2+\mu\|u(t)\|^2_{\dot{H}^{1}}\right)+\nu\|u(t)\|_{\dot{H}^{\frac{1+\alpha}{2}}}^2\leq C(M)\|u(t)\|^2_{\dot{H}^{1}}.
$$
As a consequence, we can apply Gronwall's inequality to get a global bound
$$
\|u(t)\|_{\dot{H}^{0.5}}^2+\mu\|u(t)\|^2_{\dot{H}^{1}}+\nu\int_0^t\|u(t)\|_{\dot{H}^{\frac{1+\alpha}{2}}}^2\leq C(\mu,M)e^{C(M)t}\|u_0\|^2_{H^{1}}.
$$
Now we test against $-\pax^2u$ and we conclude as in Step 1. Case (i) is proved.

\section{Proof of Theorem \ref{teo4}: Global solution if $0<\alpha<1$ and $\mu=0$}
We consider the case $s=2$, the other cases being similar. Let us write $\tilde {x}_t$ for the point where $\pax u$ reaches its maximum, \emph{i.e.}
$$
\pax u(\tilde {x}_t,t)=\max_{x} (\pax u(x,t) )=M_1(t)>0.
$$
With a similar argument as in the proof of Proposition \ref{teo1b} (see also \cite{cor2}), we have
$$
\frac{d}{dt}M_1(t)=-\pax\left(\frac{2u}{u^2+M(1-u)^2}-\frac{u^2(2u-2M(1-u))}{\left(u^2+M(1-u)^2\right)^2}\right)\pax u\bigg{|}_{x=\tilde{x}_t}-\nu\Lambda^{\alpha}\pax u\bigg{|}_{x=\tilde{x}_t}.
$$
Due to the kernel expression \eqref{1} and \eqref{normalizingconstant}, we have
$$
-\nu\Lambda^{\alpha}\pax u(\tilde{x}_t)\leq -\nu \mathcal{C}_{\alpha,1}\pax u(\tilde{x}_t).
$$

Due to the smallness choice \eqref{12} and $\|u(t)\|_{L^\infty} \le \|u_0\|_{L^\infty}$, we have
\begin{align*}
A& =\left|\pax\left(\frac{2u}{u^2+M(1-u)^2}-\frac{u^2(2u-2M(1-u))}{\left(u^2+M(1-u)^2\right)^2}\right)\right|\\
& \leq \Sigma(\gamma^*)\\
& \leq \nu\mathcal{C}_{\alpha,1}.
\end{align*}
Consequently,
\[
\frac{d}{dt}M_1(t) \leq \left[-\pax\left(\frac{2u}{u^2+M(1-u)^2}-\frac{u^2(2u-2M(1-u))}{\left(u^2+M(1-u)^2\right)^2}\right)-\nu \mathcal{C}_{\alpha,1}\right]\pax u  \leq 0.
\]
Let us write $ \underline{\tilde x}_t$ for the point where $\pax u$ reaches its minimum, \emph{i.e.}
$$
\pax u( \underline{\tilde x}_t,t)=\min_{x}\pax u(x,t)=m_1(t).
$$

As before, due to the kernel expression \eqref{1} and \eqref{normalizingconstant}, we have
$$
-\nu\Lambda^{\alpha}\pax u(\underline{\tilde x}_t)\geq -\nu \mathcal{C}_{\alpha,1}\pax u( \underline{\tilde x}_t).
$$
Consequently, with the same argument, we have (for negative $\pax u(\underline {\tilde {x}}_t,t)$)
\begin{align*}
\frac{d}{dt}m_1(t)&=-\pax\left(\frac{2u}{u^2+M(1-u)^2}-\frac{u^2(2u-2M(1-u))}{\left(u^2+M(1-u)^2\right)^2}\right)\pax u\bigg{|}_{x=\underline{\tilde x}}-\nu\Lambda^{\alpha}\pax u(\underline{\tilde x}_t)\\
&\geq -\pax u(\underline{\tilde x}_t)\left[\pax\left(\frac{2u}{u^2+M(1-u)^2}-\frac{u^2(2u-2M(1-u))}{\left(u^2+M(1-u)^2\right)^2}\right)+\nu \mathcal{C}_{\alpha,1}\right]\\
&\geq0.
\end{align*}
Hence
\begin{equation}\label{16}
\|u(t)\|_{W^{1,\infty}(\TT)}\leq \|u_0\|_{W^{1,\infty}(\TT)}.
\end{equation}
We test equation \eqref{BL} against $\pax^4u$ and integrate by parts. We have
\begin{equation}\label{13}
\frac{1}{2}\frac{d}{dt}\|u\|^2_{\dot{H}^2}=I_5+I_6,
\end{equation}
with
\begin{align}
I_5&=-\int_\TT \pax^2\left[\left(\frac{2u}{u^2+M(1-u)^2}-\frac{u^2(2u-2M(1-u))}{\left(u^2+M(1-u)^2\right)^2}\right)\pax u\right]\pax^2udx\label{14}\\
I_6&=-\nu\|u\|_{\dot{H}^{2+\frac{\alpha}{2}}}^2.\label{15}
\end{align}
Using \eqref{16}, we have
$$
I_5\leq c(u_0,M)\|u\|_{\dot{H}^2}^2,
$$
thus, applying Gronwall's inequality, we have
$$
\|u(t)\|_{H^2}^2+2\nu\int_0^t\|u(s)\|_{\dot{H}^{2+\frac{\alpha}{2}}}^2ds\leq \|u_0\|_{H^2}^2e^{c(u_0,M)t}.
$$
\section{Proof of Proposition \ref{teo5}: Finite time singularities}\label{sec:8}
First, we study the case $\nu=0$. Let us take $u_0$ such that
$$
u_0\geq0,\,\,u_0(0)=0,
$$
and
\begin{equation}\label{19}
J_0=\int_{-1}^0\frac{u_0(x)}{|x|^{\delta}}dx<\infty.
\end{equation}

We argue by contradiction: assume that we have $u(t)$ a global $C^2$ solution corresponding to $u_0$. Recalling the expression $a(x)$ given in \eqref{eqa} we define the characteristic curve $y(t)$, solution to
\begin{equation}\label{charcurve}
y'(t)=a(u(y(t),t)),\,\,y(0)=0.
\end{equation}
and
$$
v(x,t)=u(x+y(t),t).
$$
Notice that, due to \eqref{BL},
$$
\frac{d}{dt}u(y(t),t)=0.
$$
Thus
$$
u(y(t),t)=u_0(y(0))=0=v(0,t).
$$
Now we have
\begin{align}
\pat v(x)&=\pat u(x+y(t),t)+\pax v(x)a(v(0))\nonumber\\
&=\pax v(x)(a(v(0))-a(v(x)))\nonumber\\
&=-\pax\left(f(v(x))\right).\label{eqv},
\end{align}
with $f$ given by \eqref{eqf}.
For a fixed $0<\delta<1$, we define
$$
\phi(x)=|x|^{-\delta}\textbf{1}_{[-1,0]}
$$
and
\begin{equation}\label{eqJ}
J(t)=\int_\RR (v(x,t)-v(0,t))\phi(x)dx.
\end{equation}
Notice that if $J(t)$ blows up, due to the inequality
$$
J(t)= \int_{-1}^0\frac{v(x,t)-v(0,t)}{|x|^{\delta}}dx\leq \|v(t)\|_{C^{\delta}}=\|u(t)\|_{C^{\delta}},
$$
the solution forms a singularity.

Testing equation \eqref{eqv} against $\phi(x)$, we have
\begin{align*}
\frac{d}{dt}\int_{-1}^0 v(x)\phi(x)dx&=-\int_{-1}^0\pax\left(f(v(x))\right)\phi(x)dx\\
&=\int_{-1}^0f(v(x))\pax \phi(x)dx+f(v(-1))\phi(-1)\\
&=\delta \int_{-1}^0f(v(x))\frac{1}{|x|^{1+\delta}}dx+f(v(-1))\\
&=\delta \int_{-1}^0f(v(x))\phi(x)^2\frac{|x|^{2\delta}}{|x|^{1+\delta}}dx+f(v(-1))\\
&\geq \frac{\delta}{1+M} \int_{-1}^0(v(x)\phi(x))^2dx\\
&\geq \frac{\delta}{1+M} \left(\int_{-1}^0v(x)\phi(x)dx\right)^2,
\end{align*}
where we have used the positivity of the solution and Jensen's inequality.

We obtain the ODI
$$
\frac{d}{dt}J(t)\geq \frac{\delta}{1+M} J(t)^2,
$$
and the blow up of $J(t)$ in finite time $T^*=T^*(\delta,u_0,M)$.

We have proved the case $\nu=0$, but the proof of the case $0<\nu$ and $\alpha=0$ is analogous and can be easily adapted from here. 

\section{Numerical simulations}\label{sec9}
In this section we present our numerical simulations suggesting a finite time blow up in the case $\nu>0$, $0<\alpha<1$. To approximate the solution, we discretize using the Fast Fourier Transform with $N=2^{14}$ spatial nodes. The main advantage of this numerical scheme is that the differential operators are multipliers on the Fourier side. Once the spatial part has been discretized, we use a Runge-Kutta scheme to advance in the time variable. 

\begin{figure}[t!]
    \centering
    \includegraphics[scale=0.35]{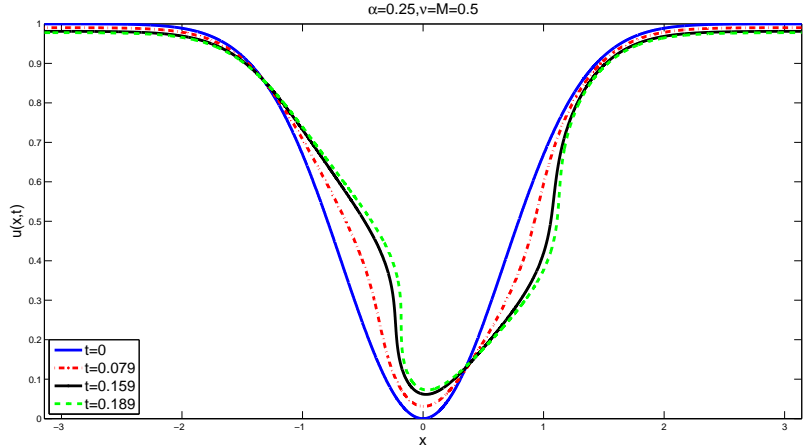}
\caption{Evolution in the case $\alpha=0.25$, $\nu=M=0.5$, $\mu=0$.}
\label{alpha025}
\end{figure}

In our simulations, we consider the initial data
\begin{equation}\label{20}
u_0(x)=1-e^{-x^2}\left(1-\frac{x^2}{\pi^2}\right),
\end{equation}
and values $M=\nu=0.5.$ and $\mu=0$. Then, we approximate the solution for \eqref{BL} for different values of the parameter $0<\alpha\leq 1$. In particular, we study four cases:
\begin{enumerate}
\item $\alpha=0.25$,
\item $\alpha=0.5$,
\item $\alpha=0.75$,
\item $\alpha=1$.
\end{enumerate} 

\begin{figure}[t!]
    \centering
    \includegraphics[scale=0.35]{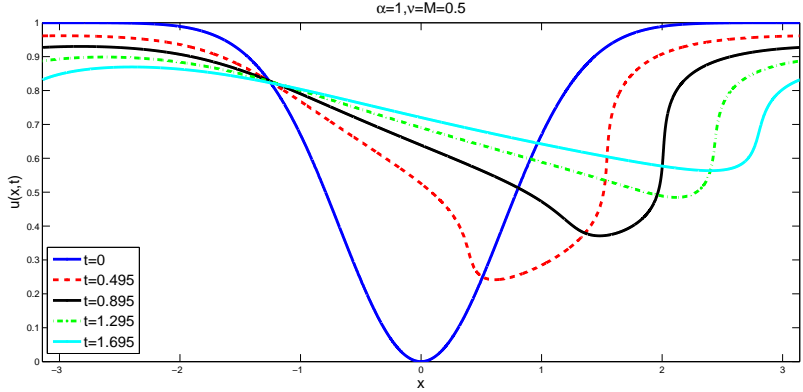}
\caption{Evolution in the case $\alpha=1$, $\nu=M=0.5$, $\mu=0$.}
\label{alpha1}
\end{figure}

Qualitatively, the evolution in the cases $\alpha=0.25$, $\alpha=0.5$ and $\alpha=0.75$ looks alike. The numerics suggests that a blow up of the derivative appears in finite time in the cases $\alpha=0.25$, $\alpha=0.5$ and $\alpha=0.75$ (see Figures \ref{alpha025} and \ref{deriv}):
$$
\limsup_{t\rightarrow T_{max}}\|\pax u(t)\|_{L^\infty}=\infty.
$$
However, in the case $\alpha=1$, the solution seems to exists globally (see Figure \ref{alpha1}).

\begin{figure}[t!]
    \centering
    \includegraphics[scale=0.35]{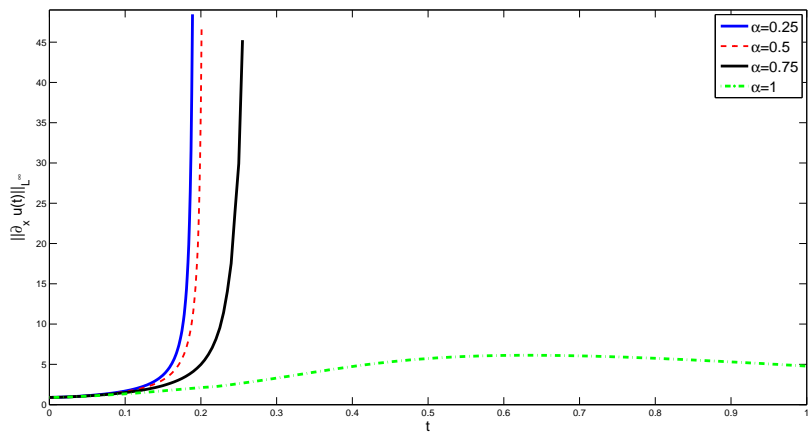}
\caption{Evolution of $\|\pax u(t)\|_{L^\infty}$, $\nu=M=0.5$, $\mu=0$.}
\label{deriv}
\end{figure}

In Figure \ref{deriv}, we plot the evolution of $\|\pax u(t)\|_{L^\infty}$. This figure shows that in the critical case $\alpha=1$, the derivative may grow for short time, even if it remains globally bounded for large times. 

Next, we add the term $\mu\Lambda^\beta \pat u$. We consider the same initial data \eqref{20} and values $M=\nu=\mu=0.5$. Then, we approximate the solution for \eqref{BL} for different values of the parameters $0<\alpha,\beta< 1$. In particular, we study three cases:
\begin{enumerate}
\item $\alpha=\beta=0.25$,
\item $\alpha=\beta=0.5$,
\item $\alpha=\beta=0.75$.
\end{enumerate} 

\begin{figure}[t!]
    \centering
    \includegraphics[scale=0.35]{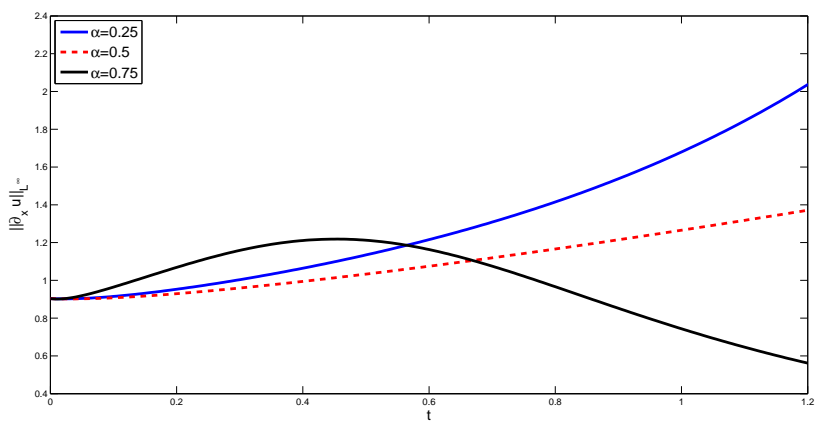}
\caption{Evolution of $\|\pax u(t)\|_{L^\infty}$, $\nu=M=\mu=0.5$.}
\label{deriv2}
\end{figure}

Interestingly, we observe (see Figure \ref{deriv2}) that even for small values of $\alpha$ and $\beta$, in the case with $\mu>0$, there is not evidence of finite time singularities.


%
%

%


\bibliographystyle{abbrv}

\end{document}